\documentclass[preprint]{elsarticle}

\usepackage[english]{babel}
\usepackage{enumerate}
\usepackage{color}
\usepackage{subfigure}
\usepackage{dsfont}
\usepackage{graphicx}
\usepackage{epstopdf}
\usepackage{amsmath}
\usepackage{mathtools}
\usepackage{amsthm}
\usepackage{amstext}
\usepackage{amssymb}
\usepackage{mathrsfs}
\usepackage{mathtools}
\usepackage{tikz}
\usetikzlibrary{arrows,positioning}

\allowdisplaybreaks

\def\Pr{\mathop{\rm Pr}}

\def\B{{\mathcal B}}

\def\P{{\mathcal P}}

\def\sPr{{\mathsf{Pr}}}

\def\sX{{\mathds X}}

\def\sU{{\mathds U}}

\def\sZ{{\mathds Z}}

\def\sU{{\mathds U}}

\theoremstyle{remark}
\newtheorem{lem}{Lemma}
\newtheorem{thm}{Theorem}

\newtheorem{assumption}{Assumption}

\theoremstyle{definition}

\newtheorem{mydef}{Definition}

\theoremstyle{remark}
\newtheorem*{remark}{Remark}

\newcommand{\R}{\mathds{R}}
\newcommand{\Zplus}{\mathds{Z}_+}
\newcommand{\N}{\mathds{N}}

\newcommand{\dd}{\mathrm{d}}

\newcommand{\sy}[1]{{\color{black} #1}}
\newcommand{\adk}[1]{{\color{black} #1}}
\newcommand{\ns}[1]{{\color{black} #1}}

\journal{Systems and Control Letters}

\begin{document}
\begin{frontmatter}
\sloppy
\title{Weak Feller Property of Non-linear Filters\tnoteref{title}}
 \tnotetext[title]{This research was supported in part by
the Natural Sciences and Engineering Research Council (NSERC) of Canada.}

 \tnotetext[title]{Some of the results in this paper are to be presented at the 2019 IEEE Conference on Decision and Control.}

\author[adk,ns]{Ali Devran Kara, Naci Saldi and Serdar Y\"uksel}

\address[adk]{Ali Devran Kara and  Serdar Y\"uksel are with the Department of Mathematics and Statistics,
     Queen's University, Kingston, ON, Canada,
     Emails: \{16adk,yuksel\}@mast.queensu.ca.}
\address[ns]{Naci Saldi is with the Department of Natural and Mathematical Sciences, Ozyegin University, Cekmekoy, Istanbul, Turkey,
Email: naci.saldi@ozyegin.edu.tr}


\begin{abstract}
Weak Feller property of controlled and control-free Markov chains lead to many desirable properties. In control-free setups this leads to the existence of invariant probability measures for compact spaces and applicability of numerical approximation methods. For controlled setups, this leads to existence and approximation results for optimal control policies. We know from stochastic control theory that partially observed systems can be converted to fully observed systems by replacing the original state space with a probability measure-valued state space, with the corresponding kernel acting on probability measures known as the non-linear filter process. Establishing sufficient conditions for the weak Feller property for such processes is a significant problem, studied under various assumptions and setups in the literature. \sy{ In this paper, we prove the weak Feller property of the non-linear filter process (i) first under weak continuity of the transition probability of controlled Markov chain and total variation continuity of its observation channel, and then, (ii) under total variation continuity of the transition probability of controlled Markov chain. The former result (i) has first appeared in Feinberg et. al. [\textit{Math. Oper. Res.} \textbf{41}(2) (2016) 656-681]. Here,  we present a concise and easy to follow alternative proof for this existing result. The latter result (ii) establishes weak Feller property of non-linear filter process under conditions which have not been previously reported in the literature.}
\end{abstract}
\end{frontmatter}

\section{Introduction}\label{section:intro}

\subsection{Preliminaries}

We start with the probabilistic setup of the partially observed Markov processes. Let $\mathds{X} \subset \mathds{R}^n$ be a Borel set in which a control-free or controlled Markov process $\{X_t,\, t \in \Zplus\}$ takes its value. Here and throughout the paper, $\Zplus$ denotes the set of non-negative integers and $\mathds{N}$ denotes the set of positive integers. Let $\mathds{U}$, the action space, be a Borel subset of some Euclidean space $\mathds{R}^p$. Let $\mathds{Y} \subset \mathds{R}^m$ be a Borel set, and let an observation channel $Q$ be defined as a stochastic kernel (regular conditional probability) from  $\mathds{X} \times \mathds{U}$ to $\mathds{Y}$ such that $Q(\,\cdot\,|x,u)$ is a probability measure on the Borel
$\sigma$-algebra ${\cal B}(\mathds{Y})$ of $\mathds{Y}$ for every $(x,u) \in \mathds{X} \times \mathds{U}$ and $Q(A|\,\cdot\,): \mathds{X} \times \mathds{U} \to [0,1]$ is a Borel measurable function for every $A \in {\cal B}(\mathds{Y})$. Let a decision maker (DM) be located at the output of an observation channel $Q$, with inputs $(X_t,U_{t-1})$ and outputs $Y_t$. An {\em admissible policy} $\gamma$ is a sequence of control functions $\{\gamma_t,\, t\in \Zplus\}$ such that $\gamma_t$ is measurable with respect to the $\sigma$-algebra generated by the information variables
\[
I_t=\{Y_{[0,t]},U_{[0,t-1]}\}, \quad t \in \mathds{N}, \quad
  \quad I_0=\{Y_0\},
\]
where
\begin{equation}
\label{eq_control}
U_t=\gamma_t(I_t),\quad t\in \Zplus
\end{equation}
are the $\mathds{U}$-valued control
actions and
\[Y_{[0,t]} = \{Y_s,\, 0 \leq s \leq t \}, \quad U_{[0,t-1]} =
  \{U_s, \, 0 \leq s \leq t-1 \}.\]
We define $\Gamma$ to be the set of all such admissible policies.

The joint distribution of the state, control, and observation processes is determined by (\ref{eq_control}) and the following system dynamics:
\[  \Pr\bigl( (X_0,Y_0)\in B \bigr) =  \int_B Q_0(dy_0|x_0) P_0(dx_0), B\in \mathcal{B}(\mathds{X}\times\mathds{Y}), \]
where $P_0$ is the prior distribution of the initial state $X_0$ and $Q_0$ is the prior control-free observation channel, and for $t\in \mathds{N}$
\begin{eqnarray*}
\label{eq_evol}
\Pr\biggl( (X_t,Y_t)\in B \, \bigg|\, (X,Y,U)_{[0,t-1]}=(x,y,u)_{[0,t-1]} \biggr)
 \\
 = \int_B Q(dy_t|x_t,u_{t-1}) \mathcal{T}(dx_t|x_{t-1}, u_{t-1}),  B\in \mathcal{B}(\mathds{X}\times
\mathds{Y}),
\end{eqnarray*}
where $\mathcal{T}(\cdot|x,u)$ is a stochastic kernel from $\mathds{X}\times \mathds{U}$ to $\mathds{X}$. This completes the probabilistic setup of the partially observed models.

\begin{remark}
Note that if $\mathds{U} = \{u\}$ is a singleton, then we recover the control-free setup. Therefore, results established in this paper also hold for control-free setup. 
\end{remark}

\ns{
Often, one is faced with an optimal control problem (or an optimal decision-making problem when control is absent in the transition dynamics). For sake of completeness, we present typical criteria in the following. Let a one-stage cost function $c:\mathds{X} \times \mathds{U}\rightarrow[0,\infty)$, which is a Borel measurable function from $\mathds{X} \times \mathds{U}$ to $[0,\infty)$, be given. Then, we denote by $W(\gamma)$ the cost function of the policy $\gamma \in \Gamma$, which can be, for instance, discounted cost or average cost criteria \cite{HeLa96}. With these definitions, the goal of the control problem is to find an optimal policy $\gamma^*$ that minimizes $W$.
}

\subsection{Markov Property of Filter Processes}

It is known that any POMP can be reduced to a completely observable Markov process \cite{Yus76}, \cite{Rhe74}, whose states are the posterior state distributions or 'beliefs` of the observer; that is, the state at time $t$ is
\begin{align}
Z_t(\,\cdot\,) := \sPr\{X_{t} \in \,\cdot\, | Y_0,\ldots,Y_t, U_0, \ldots, U_{t-1}\} \in \P(\sX). \nonumber
\end{align}
We call this equivalent process the filter process \index{Belief-MDP}. The filter process has state space $\sZ = \P(\sX)$ and action space $\sU$. Note that $\sZ$ is equipped with the Borel $\sigma$-algebra generated by the topology of weak convergence \cite{Bil99}. Under this topology, $\sZ$ is a Borel space \cite{Par67}. The transition probability of the filter process can be constructed as follows.

The joint conditional probability on next state and observation variables given the current control action and the current state of the filter process is given by
\begin{align}\label{r_kernel1}
R(B\times C|u_0,z_0) = \int_{\mathds{X}} \int_B Q(C|x_1,u_0)\mathcal{T}(dx_1|x_0,u_0)z_0(dx_0),
\end{align}
for all $B \in \B(\mathds{X})$ and $C \in \B(\mathds{Y})$. Then, the conditional distribution of the next observation variable given the current state of the filter process and the current control action is given by
\begin{align*}
P(C|u_0,z_0) = \int_{\mathds{X}} \int_{\mathds{X}} Q(C|x_1,u_0)\mathcal{T}(dx_1|x_0,u_0)z_0(dx_0),
\end{align*}
for all $C\in \B(\mathds{Y})$. Using this, we can disintegrate $R$ (see \cite[Proposition 7.27]{bertsekas78}) as follows:
\begin{align}\label{r_kernel2}
&R(B\times C|u_0,z_0) = \int_C F(B|y_1,u_0,z_0) P(dy_1|u_0,z_0) \nonumber \\
&\phantom{xxxxxxxxxxxxxxxxxxxxxxxxx}=\int_C z_1(y_1,u_0,z_0)(B) P(dy_1|u_0,z_0),
\end{align}
where $F$ is a stochastic kernel from $\mathds{Z}\times \mathds{Y}\times\mathds{U}$ to $\mathds{X}$ and the posterior distribution of $x_1$, determined by the kernel $F$, is the state variable $z_1$ of the filter process.
Then, the transition probability $\eta$ of the filter process can be constructed as follows (see also \cite{Her89}). If we define the measurable function $F(z,u,y) := F(\,\cdot\,|y,u,z) = \Pr\{X_{t+1} \in \,\cdot\, | Z_t = z, U_t = u, Y_{t+1} = y\}$ from $\mathds{Z}\times\mathds{U}\times\mathds{Y}$ to $\mathds{Z}$ and use the stochastic kernel $P(\,\cdot\, | z,u) = \Pr\{Y_{t+1} \in \,\cdot\, | Z_t = z, U_t = u\}$ from $\mathds{Z}\times\mathds{U}$ to $\mathds{Y}$, we can write $\eta$ as
\begin{align}
\eta(\,\cdot\,|z,u) = \int_{\mathds{Y}} 1_{\{F(z,u,y) \in \,\cdot\,\}} P(dy|z,u). \label{beliefK}
\end{align}

The one-stage cost function $\tilde{c}:\mathds{Z} \times \mathds{U}\rightarrow[0,\infty)$ of the filter process is given by
\begin{align}
\tilde{c}(z,u) := \int_{\sX} c(x,u) z(dx), \nonumber
\end{align}
which is a Borel measurable function. Hence, the filter process is a completely observable Markov process with the components $(\sZ,\sU,\tilde{c},\eta)$.

For the filter process, the information variables is defined as
\[
\tilde{I}_t=\{Z_{[0,t]},U_{[0,t-1]}\}, \quad t \in \mathds{N}, \quad
  \quad \tilde{I}_0=\{Z_0\}.
\]
It is a standard result that an optimal control policy of the original POMP will use the belief $Z_t$ as a sufficient statistic for optimal policies (see \cite{Yus76}, \cite{Rhe74}). More precisely, the filter process is equivalent to the original POMP  in the sense that for any optimal policy for the filter process, one can construct a policy for the original POMP which is optimal. 


\subsection{Problem Formulation}\label{problem}

Let us review two convergence notions for probability measures, and also, we define the weak Feller property of Markov decision processes. Let $(\mathds{S},d)$ be a separable metric space. A sequence $\{\mu_n,n\in\N\}$ in the set of probability measures $\mathcal{P}(\mathds{S})$ is said to converge to $\mu\in\mathcal{P}(\mathds{S})$ \emph{weakly} if
\begin{align*}
    \int_{\mathds{S}}f(x)\mu_n(dx) \to \int_{\mathds{S}}f(x)\mu(dx)
\end{align*}
for every continuous and bounded $f:\mathds{S} \to \R$. The topology of weak convergence on the set of probability measures on a separable metric space is metrizable. One such metric is the bounded-Lipschitz metric. For any two probability measures $\mu$ and $\nu$, the bounded-Lipschitz metric is defined as:
\begin{align*}
\rho_{BL}(\mu,\nu)=\sup_{\|f\|_{BL}\leq1} \biggr| \int_{\mathds{S}} f(x) \mu(dx) - \int_{\mathds{S}} f(x) \nu(dx) \biggl|
\end{align*}
where $\|f\|_{BL}=\|f\|_\infty+\sup_{x\neq y}\frac{|f(x)-f(y)|}{d(x,y)}$ and $\|f\|_\infty=\sup_{x\in\mathds{S}}|f(x)|$. Another metric that metrizes the weak topology on $\P(\mathds{S})$ is the following:
\begin{align}
\rho(\mu,\nu) = \sum_{m=1}^{\infty} 2^{-(m+1)} \biggr| \int_{\mathds{S}} f_m(x) \mu(dx) - \int_{\mathds{S}} f_m(x) \nu(dx) \biggl|, \nonumber
\end{align}
where $\{f_m\}_{m\geq1}$ is an appropriate sequence of continuous and bounded functions such that $\|f_m\|_{\infty} \leq 1$ for all $m\geq1$ (see \cite[Theorem 6.6, p. 47]{Par67}).

For probability measures $\mu,\nu\in\mathcal{P}(\mathds{S})$, the \emph{total variation} norm is given by
\begin{align*}
    \|\mu-\nu\|_{TV}&=2\sup_{B\in\mathcal{B}(\mathds{S})}|\mu(B)-\nu(B)|=\sup_{f:\|f\|_\infty \leq 1}\left|\int f(x)\mu(\dd x)-\int f(x)\nu(\dd x)\right|,
\end{align*}
where the supremum is taken over all Borel measurable real $f$ such that $\|f\|_\infty \leq 1$. A sequence $\{\mu_n,n\in\N\}$ in $\mathcal{P}(\mathds{S})$ is said to converge in total variation to $\mu\in\mathcal{P}(\mathds{S})$ if $\|\mu_n-\mu\|_{TV}\to 0$.
\begin{mydef}[\textbf{Weak Feller Property}]
We say that a Markov decision process with transition kernel $\mathcal{T}(\cdot|x,u)$ has weak Feller property if $\mathcal{T}$ is weakly continuous in $x$ and $u$; that is, if $(x_n,u_n) \rightarrow (x,u)$, then $\mathcal{T}(\cdot|x_n,u_n) \rightarrow \mathcal{T}(\cdot|x,u)$ weakly.
\end{mydef}

With this definition, we can now state the problem that we are interested in this paper. 
\begin{itemize}
\item[\textbf{(P)}] Under what conditions on the transition kernel and the observation channel of the partially observed model, the filter process has a weak Feller property?
\end{itemize}


\subsection{Significance of the Problem}

In this section, we motivate the operational (in view of engineering applications) and the mathematical importance of the problem in view of existence and invariance properties, approximations and computational results involving non-linear filters and stochastic control, and related applications involving particle filtering.

For finite-horizon problems and a large class of infinite-horizon discounted cost problems, it is a standard result that an optimal control policy will use the filter process as a sufficient statistic for optimal policies (see \cite{Yus76,Rhe74,Blackwell2}). Hence, the partially observed model and the corresponding filter process are equivalent in the sense of cost minimization. Therefore, results developed for the standard controlled Markov process problems (e.g., measurable selection criteria as summarized in \cite[Chapter 3]{HeLa96}) can be applied to the filter processes, and so, to the partially observed models. In controlled Markov processes theory, weak continuity of the transition kernel is an important condition leading to both the existence of optimal control policies for finite-horizon and infinite-horizon discounted cost problems as well as the continuity properties of the value function (see, e.g., \cite[Section 8.5]{HeLa99}).

For partially observed stochastic control problems with the average cost criterion, the conditions of existence of optimal policies stated in the literature are somewhat restrictive, with the most relaxed conditions to date being reported in \cite{borkar2000average, borkar2004further}, to our knowledge. For such average cost stochastic control problems, the weak Feller property can be utilized to obtain a direct method to establish the existence of optimal stationary (and possibly randomized) control policies. Indeed, for such problems, the convex analytic method (\cite{Manne} and \cite{Borkar2}) is a powerful approach to establish the existence of optimal policies. If one can establish the weak Feller property of the filter process, then the continuity and compactness conditions utilized in the convex program of \cite{Borkar2} would lead to the existence of optimal average cost policies.

In addition to existence of optimal policies, the weak Feller property has also recently been shown to lead to (asymptotic) consistency in approximation methods for controlled Markov processes with uncountable state and action spaces. In \cite{SaYuLi16,saldi2017asymptotic}, authors showed that optimal policies obtained from finite-model approximations to infinite-horizon discounted cost controlled Markov processes with Borel state and action spaces asymptotically achieve the optimal cost for the original problem under weak Feller property. Hence, the weak Feller property of filter process suggests that approximation results in \cite{SaYuLi16,saldi2017asymptotic}, which only require weak continuity conditions on the transition probability of a given controlled Markov model, are particularly suitable in developing approximation methods for partially observed models (through their reduction to fully observed models).

For control-free setups, the weak Feller property of the filter processes leads to the existence of invariant probability measures when the hidden Markov processes take their values in compact spaces or more general spaces under appropriate tightness/regularity conditions \cite{Lasserre,budhiraja2002,MaSt08}.

For the empirical consistency and convergence results involving the very popular particle filtering algorithms, weak Feller property of the filter processes is a commonly imposed, implicit, assumption (see e.g. \cite{ del98,crisan02}). Finally, for the study of ergodicity and asymptotic stability of nonlinear filters, weak Feller property also plays an important role (see \cite{stett89, budhiraja2000, budhiraja2002}).

\section{Main Results and Connections with the Literature}

\subsection{Statement of Main Results}

In this paper, we show the weak Feller property of the filter process under two different set of assumptions.
\begin{assumption}\label{TV_channel}
\begin{itemize}
\item[]
\item[(i)] The transition probability $\mathcal{T}(\cdot|x,u)$ is weakly continuous in $(x,u)$, i.e., for any $(x_n,u_n)\to (x,u)$, $\mathcal{T}(\cdot|x_n,u_n)\to \mathcal{T}(\cdot|x,u)$ weakly.
\item[(ii)] The observation channel $Q(\cdot|x,u)$ is continuous in total variation, i.e., for any $(x_n,u_n) \to (x,u)$, $Q(\cdot|x_n,u_n) \rightarrow Q(\cdot|x,u)$ in total variation.
\end{itemize}
\end{assumption}

\begin{assumption}\label{TV_kernel}
\begin{itemize}
\item[]
\item[(i)] The transition probability $\mathcal{T}(\cdot|x,u)$ is continuous in total variation in $(x,u)$, i.e., for any $(x_n,u_n)\to (x,u)$, $\mathcal{T}(\cdot|x_n,u_n) \to \mathcal{T}(\cdot|x,u)$ in total variation.
\item[(ii)] The observation channel $Q(\cdot|x)$ is independent of the control variable.
\end{itemize}
\end{assumption}

We now formally state the main results of our paper.

\begin{thm}[Feinberg et. al. \cite{FeKaZg14}]\label{TV_channel_thm}
Under Assumption \ref{TV_channel}, the transition probability $\eta(\cdot|z,u)$ of the filter process is weakly continuous in $(z,u)$.
\end{thm}

\begin{thm}\label{TV_kernel_thm}
Under Assumption \ref{TV_kernel}, the transition probability $\eta(\cdot|z,u)$ of the filter process is weakly continuous in $(z,u)$.
\end{thm}

Theorem \ref{TV_channel_thm} is originally due to Feinberg et. al.~\cite{FeKaZg14}. The contribution of our paper is that the proof presented here is more direct and significantly more concise. Theorem \ref{TV_kernel_thm} establishes that under Assumption \ref{TV_kernel} (with no assumptions on the measurement model), the filter process is weakly continuous. This result has not been previously reported in the literature. 

The proofs of these results are presented in Section~\ref{ProofSection}.

\adk{
\begin{remark}
We note that there are examples where Assumptions Assumption \ref{TV_channel} or \ref{TV_kernel} are not satisfied (e.g. \cite[Section 8.2]{FeKaZg14}), however the filter kernel $\eta$ is still weak Feller. We refer the reader to the result \cite[Theorem 7.1]{FeKaZg14} which establishes weak Feller property of the filter kernel under further sets of assumptions. We also note that, after the first submission of this paper to Arxiv, \cite[Theorem 4.4]{FeKaZg_19} reported a result generalizing Theorem \ref{TV_kernel_thm} to the channel $Q(\cdot|x,u)$ depending on the control variable as well and being continuous in total variation in $u$.
\end{remark}
}

\subsection{Examples}\label{exmp}

In this section, we give concrete examples for the system and observation channel models which satisfy Assumption~\ref{TV_channel} or Assumption~\ref{TV_kernel}. Suppose that the system dynamics and the observation channel are represented as follows:
\begin{align*}
X_{t+1} &= H(X_t,U_t,W_t),\\
Y_t &= G(X_t,U_{t-1},V_t),
\end{align*}
where $W_t$ and $V_t$ are i.i.d. noise processes. This is, without loss of generality, always the case; that is, one can transform the dynamics of any partially observed model into this form  (see Lemma 1.2 in \cite{gihman2012controlled}, or Lemma 3.1 of \cite{BorkarRealization}).

\begin{itemize}
\item[(i)]
Suppose that $H(x,u,w)$ is a continuous function in $x$ and $u$. Then, the corresponding transition kernel is weakly continuous. To see this, observe that, for any $c\in C_b(\mathds{X})$, we have
\begin{align*}
&\int c(x_1)\mathcal{T}(dx_1|x_0^n,u_0^n)=\int c(H(x_0^n,u_0^n,w_0))\mu(dw_0)\\
&\to \int c(H(x_0,u_0,w_0))\mu(dw_0)=\int c(x_1)\mathcal{T}(dx_1|x_0,u_0),
\end{align*}
where we use $\mu$ to denote the probability model of the noise.
\item[(ii)]\label{ornii}
Suppose that $G(x,u,v)=g(x,u)+v$, where $g$ is a continuous function and $V_t$ admits a continuous density function $\varphi$ with respect to some reference measure $\nu$. Then, the channel is continuous in total variation. Notice that under this setup, we can write $Q(dy|x,u)=\varphi(y-h(x,u)) \nu(dy)$. Hence, the density of $Q(dy|x_n,u_n)$ converges to the density of $Q(dy|x,u)$ pointwise, and so, $Q(dy|x_n,u_n)$ converges to $Q(dy|x,u)$ in total variation by Scheff\'e's Lemma \cite[Theorem 16.12]{Bil95}. Hence, $Q(dy|x,u)$ is continuous in total variation under these conditions.
\item[(iii)]
Suppose that we have $H(x,u,w)=h(x,u)+w$, where $f$ is continuous and $W_t$ admits a continuous density function $\varphi$ with respect to some reference measure $\nu$. Then, the transition probability is continuous in total variation. Again, notice that with this setup we have $\mathcal{T}(dx_1|x_0,u_0)=\varphi(x_1-h(x_0,u_0)) \nu(dx_1)$. Thus, continuity of $\varphi$ and $h$ guarantees the pointwise convergence of the densities, so we can conclude that the transition probability is continuous in total variation by again Scheff\'e's Lemma.
\end{itemize}

The analysis in the paper will provide weak Feller results for a large class of partially observed control systems as reviewed in the aforementioned examples. In particular, if the state dynamics are affected by an additive noise process which admits a continuous density, we can guarantee weak Feller property of the filter process by means of Theorem~\ref{TV_kernel_thm} without referring to the noise model of the observation channel.

\subsection{Comparison with the Prior Results}

Weak Feller property of the control-free transition probability of the filter processes has been established using different approaches and different conditions. In \cite{budhiraja2000} it has been shown that, for continuous-time filter processes, if the signal process (state process of the partially observed model) is weak Feller and the measurement channel is an additive channel in the form $Y_t=\int_0^t h(X_u)du+V_t$, where $h$ is assumed to be continuous and possibly unbounded and $V_t$ is a standard Wiener process, then the filter process itself is also weak Feller. In \cite{budhiraja2002}, the authors study the discrete-time filter processes, where the state process noise may not be independent of the observation process noise; it has been shown that if the observation model is additive in the form $Y_t=h(X_t)+V_t$, where $h$ is assumed to be continuous and $V_t$ is an i.i.d. noise process which admits a continuous and bounded density function, then the observation and filter processes $(Y_t, Z_t)$ are jointly weak Feller. In \cite{stett89}, the weak Feller property of the filter process has been shown for both discrete and continuous time setups when the channel is additive, $Y_t=h(X_t)+V_t$, where $h$ is bounded and continuous and $V_t$ is an i.i.d. noise process with a continuous, bounded and positive density function. 

\cite {del98,crisan02} have studied the consistency of the particle filter methods where the weak Feller property of the filter process has been used to establish the convergence results. In \cite{del98}, it has been shown that the filter process is weak Feller under the assumption that the transition probability of the partially observed system is weak Feller and the measurement channel is an additive channel in the form $Y_t=h(X_t)+V_t$, where $h$ is assumed to be continuous and $V_t$ is an i.i.d. noise process, which admits a continuous and bounded density function with respect to the Lebesgue measure. In \cite{crisan02}, the weak Feller property of the filter process has been established under the assumption that the measurement channel admits a continuous and bounded density function with respect to the Lebesgue measure; i.e., the channel can be written in the following form: $Q(y\in A|x)=\int_A g(x,y)dy$ for any$A \in \B(\mathds{Y})$ and for any $x \in \mathds{X}$, where $g$ is a continuous and bounded function.

Weak Feller property of the controlled transition probability of the filter processes has been established, in the most general case to date, by Feinberg et.al.~\cite{FeKaZg14}. Under the assumption that the measurement channel is continuous in total variation and the transition kernel of the partially observed model is weak Feller, the authors have established the weak Feller property of the transition probability of the filter process. In Section~\ref{comp}, we will give a detailed discussion on the methods used by Feinberg et.al.~\cite{FeKaZg14}, and also, we will compare their approach with ours.


As reviewed above, the prior literature often used the additive channel model $Y_t=h(X_t,U_{t-1})+V_t$ with various regularity assumptions on $h$ and the noise model $V_t$. When the observation channel is additive $Y_t=h(X_t,U_{t-1})+V_t$, where $h$ is continuous and $V_t$ admits a continuous density function with respect to some measure $\mu$, one can show that the channel also admits a continuous density function, i.e., $Q(y\in A|x,u)=\int_A g(x,u,y) \mu(dy)$ for any $A \in \B(\mathds{Y})$ and for any $(x,u) \in \mathds{X} \times \mathds{U}$. When the observation channel has a continuous density function, the pointwise convergence of the density functions implies the total variation convergence by Scheff\'e's Lemma \cite[Theorem 16.12]{Bil95}. Thus, $g(x_k,u_k,y)\to g(x,u,k)$ for some $(x_k,u_k) \to (x,u)$ implies that $Q(\cdot|x_k,u_k)\to Q(\cdot|x,u)$ in total variation, i.e., the observation channel is continuous in total variation. 

In the following, we develop a relationship between the total variation continuity of the channel (as required by  \cite{FeKaZg14} and in our Theorem \ref{TV_channel_thm}) and the more restrictive density conditions on the measurement channels presented in the prior works \cite{budhiraja2000,budhiraja2002, crisan02, del98}.

In the theorem below, we show that having a continuous density is \emph{almost} equivalent to the condition that the observation channel is continuous in total variation. 
To our knowledge, it is the first result in the literature making the connection between channels which are continuous in total variation and channels which admit a density function.

\begin{thm}\label{aux}
Suppose that the observation channel $Q(dy|x,u)$ is continuous in total variation. Then, for any $(z,u) \in \mathds{Z}\times\mathds{U}$, we have, $\mathcal{T}(\cdot|z,u)$-a.s., that  $Q(dy|x,u) \ll P(dy|u,z)$ and
\begin{align*}
Q(dy|x,u) = g(x,u,y)P(dy|z,u)
\end{align*}
for a measurable function $g$, which satisfies for any $A \in \B(\mathds{Y})$ and for any $x_k \to x$
\begin{align*}
\int_A |g(x_k,u,y)-g(x,u,y)|P(dy|z,u)\to 0.
\end{align*}
\end{thm}
\begin{proof}
Fix any $(z,u)$. We first show that $Q(dy|x,u)\ll P(dy|z,u)$, $\mathcal{T}(\cdot|z,u)$-a.s.. Note that $Q(dy|x,u)\ll P(dy|z,u)$ if and only if, for all $\varepsilon > 0$, there exits $\delta > 0$ such that $Q(A|x,u) < \varepsilon$ whenever $P(A|z,u) < \delta$. For each $n\geq1$, let $K_n \subset \mathds{X}$ be compact such that $\mathcal{T}(K_n|z,u) > 1-/3n$. As $Q(dy|x,u)$ is continuous in total variation norm, the image of $K_n \times \{u\}$ under $Q(dy|x,u)$ is compact in $\P(\mathds{Y})$. Hence, there exist $\{\nu_1,\ldots,\nu_l\} \subset \P(\mathds{Y})$ such that
\begin{align}
\max_{x \in K_n} \min_{i=1,\ldots,l} \|Q(\cdot|x,u) - \nu_i\|_{TV} < 1/3n. \nonumber
\end{align}
Define the following stochastic kernel $\nu_{n}(\cdot|x,u) = \arg\min_{\nu_i} \|Q(\cdot|x,u) - \nu_i\|_{TV}$. Then, we define $P_{n}(\cdot|z,u) = \int_{\mathds{X}} \nu_n(\cdot|x,u) \mathcal{T}(dx|z,u)$. One can prove that $\|P(\cdot|z,u) - P_n(\cdot|z,u)\|_{TV} < 1/n$. Moreover, since $P_n(\cdot|z,u)$ is a mixture of finite probability measures $\{\nu_1,\ldots,\nu_l\}$, we have that $\nu_{n}(\cdot|x,u) \ll P_n(\cdot|z,u)$ for all $x \in C_n$, where $\mathcal{T}(C_n|z,u)=1$. Let $C = \bigcap_{n} C_n$, and so, $\mathcal{T}(C|z,u)=1$. We claim that if $x \in C$, then $Q(dy|x,u)\ll P(dy|z,u)$, which completes the proof of the first statement. To prove the claim, fix any $\varepsilon > 0$ and choose $n\geq1$ such that $\varepsilon > 2/3n$. Then, there exists $\delta > 0$ such that $\nu_{n}(A|x,u) < \varepsilon/2$ whenever $P_{n}(A|z,u) < \delta$. This implies that $Q(A|x,u) < \varepsilon$ whenever $P(A|z,u) < \delta + 1/n$. Hence, $Q(dy|x,u)\ll P(dy|z,u)$.

To show the second claim, for any $A \in \B(\mathds{Y})$ and for any $x_k \to x$ , we define
\[A_+^{(k)}:=\{y \in A: g(x_k,u,y)>g(x,u,y)\},\]
\[A_-^{(k)}:=\{y \in A: g(x_k,u,y)<g(x,u,y)\}.\]
With these sets, we have
\begin{align*}
 &\int_{A} |g(x_k,u,y)-g(x,u,y)| P(dy|z,u) \\
& = \int_{A_+^{(k)}} g(x_k,u,y)P(dy|z,u)-\int_{A_+^{(k)}}g(x,u,y) P(dy|z,u) \\
&\phantom{xx}+ \int_{A_-^{(k)}} g(x,u,y)P(dy|z,u)-\int_{A_-^{(k)}}g(x_k,u,y) P(dy|z,u)\\
&\leq |Q(A_+^{(k)}|x_k,u)-Q(A_+^{(k)}|x,u)| + |Q(A_-^{(k)}|x_k,u)-Q(A_-^{(k)}|x,u)| \\
&\leq 2 \|Q(\cdot|x_k,u) - Q(\cdot|x,u)\|_{TV} \to 0.
\end{align*}

\end{proof}
We again emphasize that weak Feller property of the filter process under Assumption~\ref{TV_channel} has been first established by \cite{FeKaZg14} using different method compared to ours. Our method is significantly more concise and direct. It is also important to note that Assumption~\ref{TV_kernel} completely eliminates any restriction on the observation channel to establish the weak Feller property of filter process. This relaxation is quite important in practice since modelling the noise on the observation channel in control problems is quite cumbersome, and in general, infeasible. But, in many problems that arise in practice, the transition probability has a continuous density with respect to some reference measure, which directly implies, via Scheff\'e's Lemma, the total variation continuity of the transition probability. We also note that the weak Feller property under only Assumption~\ref{TV_kernel}-(i) cannot be established. Indeed, Example~4.1 of \cite{FeKaZg14} shows that the total variation continuity assumption on the observation channel cannot be relaxed even when the transition kernel is continuous in total variation to prove weak Feller property of the filter process under a controlled observation channel model: A careful look at the counterexample shows that it indeed uses the discontinuity of the observation channel in the control action to prove that the filter process cannot be weak Feller when the observation channel is not continuous in total variation and the transition kernel is continuous in total variation.


\section{Proofs}\label{ProofSection}

The following result will play a key role for the proofs of main results. The proof can be found on the appendix.

\begin{lem}\label{kernel_robust}
Let $\mathds{X}$ be a Borel space. Suppose that we have a family of uniformly bounded real Borel measurable functions $\{f_{n,\lambda}\}_{n\geq1,\lambda\in \Lambda}$ and $\{f_{\lambda}\}_{\lambda \in \Lambda}$, for some set $\Lambda$. If, for any $x_n \to x$ in $\mathds{X}$, we have
\begin{align}\label{f1}
&\lim_{n\rightarrow\infty}\sup_{\lambda \in \Lambda}|f_{n,\lambda}(x_n)-f_{\lambda}(x)| = 0
\end{align}
\begin{align}\label{f2}
&\lim_{n\rightarrow\infty}\sup_{\lambda \in \Lambda}|f_{\lambda}(x_n)-f_{\lambda}(x)| = 0,
\end{align}
then, for any $\mu_n \to \mu$ weakly in $\P(\mathds{X})$, we have
\begin{align*}
\lim_{n\rightarrow\infty}\sup_{\lambda \in \Lambda}\bigg|\int_{\mathds{X}} f_{n,\lambda}(x)\mu_n(dx)-\int_{\mathds{X}} f_{\lambda}(x)\mu(dx)\bigg| = 0.
\end{align*}
\end{lem}

\ns{ 
In Theorem~\ref{TV_channel_thm} and Theorem~\ref{TV_kernel_thm}, we need to show that, for every $(z_0^n,u_n) \to (z_0,u)$ in $\mathds{Z} \times \mathds{U}$, we have
\begin{align*}
&\sup_{\|f\|_{BL}\leq1}\bigg|\int_{\mathds{Z}} f(z_1)\eta(dz_1|z_0^n,u_n)- \int_\mathds{Z} f(z_1)\eta(dz_1|z_0,u)\bigg|\to 0,
\end{align*}
where we equip $\mathds{Z}$ with the metric $\rho$ to define bounded-Lipschitz norm $\|f\|_{BL}$ of any Borel measurable function $f:\mathds{Z}\rightarrow \mathds{R}$. We can equivalently write this as
\begin{align}
&\sup_{\|f\|_{BL}\leq1}\bigg|\int_{\mathds{Y}} f(z_1(z_0^n,u_n,y_1))P(dy_1|z_0^n,u_n) - \int_{\mathds{Y}} f(z_1(z_0,u,y_1))P(dy_1|z_0,u)\bigg|\to 0. \label{first_step}
\end{align}
The term in equation (\ref{first_step}) can be upper bounded as follows:
\begin{align}\label{second_step}
&\sup_{\|f\|_{BL}\leq1} \bigg|\int_{\mathds{Y}} f(z_1(z_0^n,u_n,y_1))P(dy_1|z_0^n,u_n)- \int_{\mathds{Y}} f(z_1(z_0,u,y_1))P(dy_1|z_0,u) \bigg|\nonumber\\
&\leq\sup_{\|f\|_{BL}\leq1} \bigg| \int_{\mathds{Y}} f(z_1(z_0^n,u_n,y_1))P(dy_1|z_0^n,u_n) - \int_{\mathds{Y}} f(z_1(z_0^n,u_n,y_1))P(dy_1|z_0,u) \bigg|\nonumber\\
&\qquad+ \sup_{\|f\|_{BL}\leq1} \int_{\mathds{Y}} \big| f(z_1(z_0^n,u_n,y_1)) -f(z_1(z_0,u,y_1)) \big| P(dy_1|z_0,u)\nonumber\\
&\leq \|P(\cdot|z_0^n,u_n) - P(\cdot|z_0,u)\|_{TV}\nonumber\\
&\qquad+ \sup_{\|f\|_{BL}\leq1}\int_{\mathds{Y}} \big| f(z_1(z_0^n,u_n,y_1)) -f(z_1(z_0,u,y_1)) \big|P(dy_1|z_0,u),
\end{align}
where, in the last inequality, we have used $\|f\|_{\infty} \leq \|f\|_{BL} \leq 1$. To prove that (\ref{second_step}) (and so (\ref{first_step})) goes to $0$, it is sufficient to establish the following results:
\begin{itemize}
\item[(i)] $P(dy_1|z_0,u_0)$ is continuous in total variation,
\item[(ii)] $\lim_{n\rightarrow\infty} \int_{\mathds{Y}} \rho(z_1(z_0^n,u_n,y_1),z_1(z_0,u,y_1))P(dy_1|z_0,u) = 0$ as $(z_0^n,u_n) \rightarrow (z_0,u)$.
\end{itemize}
Indeed, suppose that (i) and (ii) hold. Then, the first term in (\ref{second_step}) goes to $0$ as $P(\cdot|z_0,u)$ is continuous in total variation. For the second term in (\ref{second_step}), we have
\begin{align*}
& \sup_{\|f\|_{BL}\leq1}\int_{\mathds{Y}} \big| f(z_1(z_0^n,u_n,y_1))-f(z_1(z_0,u,y_1)) \big| P(dy_1|z_0,u)\\
&\leq  \int_{\mathds{Y}} \rho(z_1(z_0^n,u_n,y_1),z_1(z_0,u,y_1))P(dy_1|z_0,u) \\
&\rightarrow 0 \text{ } \text{as $n \rightarrow \infty$}\text{      }\,\,\,\,(\text{by (ii)}). 
\end{align*}
Therefore, to complete the proof of Theorem~\ref{TV_channel_thm} and Theorem~\ref{TV_kernel_thm}, we will prove (i) and (ii). 
}

\subsection{Proof of Theorem \ref{TV_channel_thm}}

We first prove (i); that is, $P(dy_1|z_0,u)$ is continuous in total variation. To this end, let $(z_0^n,u_n) \rightarrow (z_0,u)$. Then, we write
\begin{align*}
&\sup_{A \in \B(\mathds{Y})} \big|P(A|z_0^n,u_n) - P(A|z_0,u)\big|\\
&=\sup_{A \in \B(\mathds{Y})} \bigg|\int_{\mathds{X}} Q(A|x_1,u_n)\mathcal{T}(dx_1|z_0^n,u_n) - \int_{\mathds{X}} Q(A|x_1,u)\mathcal{T}(dx_1|z_0,u) \bigg|,
\end{align*}
where $\mathcal{T}(dx_1|z_0^n,u_n) \coloneqq \int_{\mathds{X}} \mathcal{T}(dx_1|x_0,u_n)z_0^n(dx_0)$. Note that, by Lemma \ref{kernel_robust}, we can show that $\mathcal{T}(dx_1|z_0^n,u_n) \to \mathcal{T}(dx_1|z_0,u)$ weakly. Indeed, if $g \in C_b(\mathds{X})$, then we define $r_n(x_0) = \int_{\mathds{X}} g(x_1) \mathcal{T}(dx_1|x_0,u_n)$ and $r(x_0) = \int_{\mathds{X}} g(x_1) \mathcal{T}(dx_1|x_0,u)$. Since $\mathcal{T}(dx_1|x_0,u)$ is weakly continuous, we have $r_n(x_0^n) \rightarrow r(x_0)$ when $x_0^n\rightarrow x_0$. Hence, by Lemma \ref{kernel_robust}, we have
\begin{align}
\lim_{n\rightarrow\infty} \biggl| \int_{\mathds{X}} r_n(x_0) z_0^n(dx_0) - \int_{\mathds{X}} r(x_0) z_0(dx_0) \biggr| = 0. \nonumber
\end{align}
Hence, $\mathcal{T}(dx_1|z_0^n,u_n) \to \mathcal{T}(dx_1|z_0,u)$ weakly. Moreover, the families of functions $\{Q(A|\,\cdot\,,u_n)\}_{n\geq1,A\in \B(\mathds{Y})}$ and $\{Q(A|\,\cdot\,,u)\}_{A\in\B(\mathds{Y})}$ satisfy the conditions of Lemma \ref{kernel_robust} as $Q$ is continuous in total variation distance. Therefore, Lemma \ref{kernel_robust} yields that
\begin{align*}
&\lim_{n\rightarrow\infty} \sup_{A \in \B(\mathds{Y})} \bigg|\int_{\mathds{X}} Q(A|x_1,u_n)\mathcal{T}(dx_1|z_0^n,u_n) - \int_{\mathds{X}} Q(A|x_1,u)\mathcal{T}(dx_1|z_0,u) \bigg| = 0.
\end{align*}
Thus, $P(dy_1|z_0,u)$ is continuous in total variation.

To prove (ii), we write
\begin{align*}
&\int_{\mathds{Y}} \rho(z_1(z_0^n,u_n,y_1),z_1(z_0,u,y_1))P(dy_1|z_0,u) \\
&= \int_{\mathds{Y}} \sum_{m=1}^{\infty} 2^{-m+1} \bigg| \int_{\mathds{X}} f_m(x_1) z_1(z_0^n,u_n,y_1)(dx_1) \nonumber\\
&\phantom{xxxxxxxxxxxxxxxxx}- \int_{\mathds{X}} f_m(x_1) z_1(z_0,u,y_1)(dx_1)\bigg|P(dy_1|z_0,u) \\
&=  \sum_{m=1}^{\infty} 2^{-m+1} \int_{\mathds{Y}} \bigg| \int_{\mathds{X}} f_m(x_1) z_1(z_0^n,u_n,y_1)(dx_1) \nonumber \\
&\phantom{xxxxxxxxxxxxxxxxx}- \int_{\mathds{X}} f_m(x_1) z_1(z_0,u,y_1)(dx_1)\bigg|P(dy_1|z_0,u),
\end{align*}
where we have used Fubini's theorem with the fact that $\sup_{m} \|f_m\|_{\infty} \leq 1$. For each $m$, let us define
\begin{align}\label{I_+}
&I_+^{(n)}:=\biggl\{y_1 \in \mathds{Y}: \int_{\mathds{X}} f_m(x_1) z_1(z_0^n,u_n,y_1)(dx_1) > \int_{\mathds{X}} f_m(x_1) z_1(z_0,u,y_1)(dx_1) \biggr\}\nonumber\\
&I_-^{(n)}:=\biggl\{y_1 \in \mathds{Y}: \int_{\mathds{X}} f_m(x_1) z_1(z_0^n,u_n,y_1)(dx_1) \leq \int_{\mathds{X}} f_m(x_1) z_1(z_0,u,y_1)(dx_1) \biggr\}.
\end{align}
Then, we can write
\begin{align*}
& \int_{\mathds{Y}} \bigg| \int_{\mathds{X}} f_m(x_1) z_1(z_0^n,u_n,y_1)(dx_1)  -\int_{\mathds{X}} f_m(x_1) z_1(z_0,u,y_1)(dx_1)\bigg|P(dy_1|z_0,u)\\
& =\int_{I_+^{(n)}} \biggl(\int_{\mathds{X}} f_m(x_1) z_1(z_0^n,u_n,y_1)(dx_1) - \int_{\mathds{X}} f_m(x_1) z_1(z_0,u,y_1)(dx_1)\biggr)P(dy_1|z_0,u)\\
& +\int_{I_-^{(n)}} \biggl(\int_{\mathds{X}} f_m(x_1) z_1(z_0,u,y_1)(dx_1) - \int_{\mathds{X}} f_m(x_1) z_1(z_0^n,u_n,y_1)(dx_1)\biggr)P(dy_1|z_0,u).
\end{align*}
In the sequel, we only consider the term with the set $I_+^{(n)}$. The analysis for the other one follows from the same steps. We have
\begin{align*}
& \int_{I_+^{(n)}} \bigg(\int_{\mathds{X}} f_m(x_1) z_1(z_0^n,u_n,y_1)(dx_1) - \int_{\mathds{X}} f_m(x_1) z_1(z_0,u,y_1)(dx_1)\bigg)P(dy_1|z_0,u)\\
&\leq  \int_{I_+^{(n)}} \int_{\mathds{X}} f_m(x_1) z_1(z_0^n,u_n,y_1)(dx_1)P(dy_1|z_0,u)\\
&\qquad\quad- \int_{I_+^{(n)}} \int_{\mathds{X}} f_m(x_1) z_1(z_0^n,u_n,y_1)(dx_1)P(dy_1|z_0^n,u_n)\\
&+ \int_{I_+^{(n)}} \int_{\mathds{X}} f_m(x_1) z_1(z_0^n,u_n,y_1)(dx_1)P(dy_1|z_0^n,u_n)\\
&\qquad\quad- \int_{I_+^{(n)}} \int_{\mathds{X}} f_m(x_1) z_1(z_0,u,y_1)(dx_1)P(dy_1|z_0,u)\\
&\leq \|P(dy_1|z_0,u)-P(dy_1|z_0^n,u_n)\|_{TV}\\
&\quad+\int_{\mathds{X}} \int_{I_+^{(n)}} f_m(x_1) Q(dy_1|x_1,u_n) \mathcal{T}(dx_1|z_0^n,u_n) \nonumber \\
&\phantom{xxxxxxxxxxxxxxxxxxxxxxxx}-\int_{\mathds{X}} \int_{I_+^{(n)}} f_m(x_1) Q(dy_1|x_1,u) \mathcal{T}(dx_1|z_0,u),
\end{align*}
where we have used $\|f_m\|_{\infty} \leq 1$ in the last inequality. The first term above goes to $0$ since $P(dy_1|z_0,u)$ is continuous in total variation. For the second term, we use Lemma~\ref{kernel_robust}. Indeed, families of functions $\{ f_m(\cdot) Q(A|\,\cdot\,,u_n): n\geq1, A \in \B(\mathds{Y})\}$ and $\{ f_m(\cdot) Q(A|\,\cdot\,,u): A \in \B(\mathds{Y})\}$ satisfy the conditions in Lemma~\ref{kernel_robust} as $Q$ is continuous in total variation. Hence, the second term converges to $0$ by Lemma~\ref{kernel_robust} since $\mathcal{T}(dx_1|z_0^n,u_n) \rightarrow \mathcal{T}(dx_1|z_0,u)$ weakly. Hence, for each $m$, we have
\begin{align*}
&\lim_{n\rightarrow\infty} \int_{\mathds{Y}} \bigg| \int_{\mathds{X}} f_m(x_1) z_1(z_0^n,u_n,y_1)(dx_1)\\
&\phantom{xxxxxxxxxxxxxxxx} - \int_{\mathds{X}} f_m(x_1) z_1(z_0,u,y_1)(dx_1)\bigg|P(dy_1|z_0,u) = 0. 
\end{align*}
By the dominated convergence theorem, we then have
\begin{align*}
&\lim_{n\rightarrow\infty}\int_{\mathds{Y}} \rho(z_1(z_0^n,u_n,y_1),z_1(z_0,u,y_1))P(dy_1|z_0,u)\\
&\leq  \sum_{m=1}^{\infty} 2^{-m+1} \lim_{n\rightarrow\infty} \int_{\mathds{Y}} \bigg| \int_{\mathds{X}} f_m(x_1) z_1(z_0^n,u_n,y_1)(dx_1)\\
&\qquad- \int_{\mathds{X}} f_m(x_1) z_1(z_0,u,y_1)(dx_1)\bigg|P(dy_1|z_0,u) = 0.
\end{align*}
This establishes (ii), which completes the proof together with (i).

\subsection{Proof of Theorem \ref{TV_kernel_thm}}

We first show (i); that is, $P(dy_1|z_0,u_0)$ is continuous total variation. Let $(z_0^n,u_n) \rightarrow (z_0,u)$. Then, we have
\begin{align*}
&\sup_{A \in \B(\mathds{Y})}|P(A|z_0^n,u_n) - P(A|z_0,u)|\\
&=\sup_{A \in \B(\mathds{Y})} \bigg|\int_{\mathds{X}} \int_{\mathds{X}} Q(A|x_1)\mathcal{T}(dx_1|x_0,u_n)z_0^n(dx_0) \\
&\phantom{xxxxxxxxxxxxx}- \int_{\mathds{X}} \int_{\mathds{X}} Q(A|x_1)\mathcal{T}(dx_1|x_0,u)z_0(dx_0) \bigg|.
\end{align*}
For each $A \in \B(\mathds{Y})$ and $n\geq1$, we define
$$f_{n,A}(x_0) = \int_{\mathds{X}} Q(A|x_1)\mathcal{T}(dx_1|x_0,u_n)$$ and $$f_{A}(x_0) = \int_{\mathds{X}} Q(A|x_1)\mathcal{T}(dx_1|x_0,u).$$ Then, for all $x_0^n \to x_0$, we have
\begin{align*}
&\lim_{n\rightarrow\infty} \sup_{A \in \B(\mathds{Y})} |f_{n,A}(x_0^n) - f_{A}(x_0)|\\
& =\lim_{n\rightarrow\infty} \sup_{A \in \B(\mathds{Y})} \bigg|\int_{\mathds{X}} Q(A|x_1)\mathcal{T}(dx_1|x_0^n,u_n) -\int_{\mathds{X}} Q(A|x_1)\mathcal{T}(dx_1|x_0,u) \bigg|\\
&\leq \lim_{n\rightarrow\infty} \|\mathcal{T}(dx_1|x_0^n,u_n)-\mathcal{T}(dx_1|x_0,u)\|_{TV} = 0
\intertext{and}
&\lim_{n\rightarrow\infty} \sup_{A \in \B(\mathds{Y})} |f_{A}(x_0^n) - f_{A}(x_0)| \\
&=\lim_{n\rightarrow\infty} \sup_{A \in \B(\mathds{Y})} \bigg|\int_{\mathds{X}} Q(A|x_1)\mathcal{T}(dx_1|x_0^n,u) -\int_{\mathds{X}} Q(A|x_1)\mathcal{T}(dx_1|x_0,u) \bigg|\\
&\leq \lim_{n\rightarrow\infty} \|\mathcal{T}(dx_1|x_0^n,u)-\mathcal{T}(dx_1|x_0,u)\|_{TV} = 0.
\end{align*}
Then, by Lemma \ref{kernel_robust}, we have
\begin{align*}
&\lim_{n\rightarrow\infty} \sup_{A \in \B(\mathds{Y})} \bigg| \int_{\mathds{X}} f_{n,A}(x_0) z_0^n(dx_0)- \int_{\mathds{X}} f_{A}(x_0) z_0(dx_0) \bigg| \\
&=\lim_{n\rightarrow\infty} \sup_{A \in \B(\mathds{Y})} \bigg| \int_{\mathds{X}} \int_{\mathds{X}} Q(A|x_1)\mathcal{T}(dx_1|x_0,u_n)z_0^n(dx_0)\\
&\phantom{xxxxxxxxxxxxxx} - \int_{\mathds{X}} \int_{\mathds{X}} Q(A|x_1)\mathcal{T}(dx_1|x_0,u)z_0(dx_0) \bigg| \\
&= 0.
\end{align*}
Hence, $P(dy_1|z_0,u_0)$ is continuous in total variation.

Now, we show (ii); that is, for any $(z_0^n,u_n) \rightarrow (z_0,u)$,  we have
\begin{align}
\lim_{n\rightarrow\infty} \int_{\mathds{Y}} \rho(z_1(z_0^n,u_n,y_1),z_1(z_0,u,y_1))P(dy_1|z_0,u) = 0. \nonumber
\end{align}
From the proof of Theorem~\ref{TV_channel_thm}, it suffices to show that
\begin{align}
&\lim_{n\rightarrow\infty} \int_{\mathds{X}} \int_{I_+^{(n)}} f_m(x_1) Q(dy_1|x_1)\mathcal{T}(dx_1|z_0^n,u_n) \nonumber\\
&\phantom{xxxxxxxxxxx} -\int_{\mathds{X}} \int_{I_+^{(n)}} f_m(x_1) Q(dy_1|x_1) \mathcal{T}(dx_1|z_0,u) = 0. \label{imp}
\end{align}
Indeed, we have
\begin{align*}
&\biggl| \int_{\mathds{X}} \int_{I_+^{(n)}} f_m(x_1)Q(dy_1|x_1)\mathcal{T}(dx_1|z_0^n,u_n) -\int_{\mathds{X}} \int_{I_+^{(n)}} f_m(x_1)Q(dy_1|x_1)\mathcal{T}(dx_1|z_0,u) \biggr|\\
&\leq \biggl| \int_{\mathds{X}^2} f_m(x_1) Q(I_+^{(n)}|x_1) \mathcal{T}(dx_1|x_0,u_n) z_0^n(dx_0)\\
&\phantom{xxxxxxxxxx} - \int_{\mathds{X}^2} f_m(x_1) Q(I_+^{(n)}|x_1) \mathcal{T}(dx_1|x_0,u) z_0^n(dx_0) \biggr| \\
&\qquad+ \biggl| \int_{\mathds{X}^2} f_m(x_1) Q(I_+^{(n)}|x_1) \mathcal{T}(dx_1|x_0,u) z_0^n(dx_0) \\
&\phantom{xxxxxxxxx}- \int_{\mathds{X}^2} f_m(x_1) Q(I_+^{(n)}|x_1) \mathcal{T}(dx_1|x_0,u) z_0(dx_0) \biggr| \\
&\leq \int_{\mathds{X}} \|\mathcal{T}(dx_1|x_0,u_n)-\mathcal{T}(dx_1|x_0,u)\|_{TV} z_0^n(dx_0) \\
&\qquad+\biggl| \int_{\mathds{X}^2} f_m(x_1) Q(I_+^{(n)}|x_1) \mathcal{T}(dx_1|x_0,u) z_0^n(dx_0)\\
&\phantom{xxxxxxxxxxx}- \int_{\mathds{X}^2} f_m(x_1) Q(I_+^{(n)}|x_1) \mathcal{T}(dx_1|x_0,u) z_0(dx_0) \biggr|,
\end{align*}
where we have used $\sup_{n\geq1} \sup_{x_1 \in \mathds{X}} \big|f_m(x_1) Q(I_+^{(n)}|x_1)\big| \leq 1$ in the last inequality. If we define $r_n(x_0) = \|\mathcal{T}(dx_1|x_0,u_n)-\mathcal{T}(dx_1|x_0,u)\|_{TV}$, then $r_n(x_0^n) \rightarrow 0$ whenever $x_0^n \rightarrow x_0$. Then, the first term converges to $0$ by Lemma~\ref{kernel_robust}
as $z_0^n \rightarrow z_0$ weakly. The second term also converges to $0$ by Lemma~\ref{kernel_robust}, since $\{\int_{\mathds{X}} f(x_1) Q(I_+^{(n)}|x_1) \mathcal{T}(dx_1|\cdot,u): n\geq1\}$ is a family of uniformly bounded and equicontinuous functions by total variation continuity of $\mathcal{T}(dx_1|x_0,u)$. This proves (ii) and completes the proof together with (i).

\section{A Technical Generalization}\label{comp}

In this section, we prove the weak Feller property of the filter process under more general condition than those in Theorem~\ref{TV_channel_thm} and Theorem~\ref{TV_kernel_thm}.
But, we note that it is indeed generally infeasible to establish this condition without imposing assumptions similar to the assumptions in Theorem~\ref{TV_channel_thm} and Theorem~\ref{TV_kernel_thm}. Therefore, although this condition is more general than those in Theorem~\ref{TV_channel_thm} and Theorem~\ref{TV_kernel_thm}, this generalization is not excessively important in practice.

We first note that our proof technique brings to light the main ingredients that is necessary to prove the weak Feller property of the filter process via the item (i) and eq. (\ref{imp}); that is,
\begin{itemize}
\item $P(dy_1|z_0,u_0)$ is continuous in total variation,
\item $\displaystyle
\lim_{n\rightarrow\infty} \int_{\mathds{X}} f_m(x_1) Q(I_+^{(n)}|x_1)\mathcal{T}(dx_1|z_0^n,u_n)
-\int_{\mathds{X}} f_m(x_1) Q(I_+^{(n)}|x_1) \mathcal{T}(dx_1|z_0,u) = 0.
$
\end{itemize}
This observation suggests the following condition that generalize the conditions in our previously stated main results. Let $\mathds{F} = \{f_m\}_{m\geq1} \subset C_b(\mathds{X})$ be a countable set of continuous and bounded functions such that $\|f_m\|_{\infty} \leq 1$ for all $m\geq1$, $1_{\mathds{X}} \in \mathds{F}$, and $\mathds{F}$ metrizes the weak topology on $\mathcal{P}(\mathds{X})$ via the metric $\rho$ introduced in Section~\ref{problem}. Then, we state the following assumption:
\begin{itemize}
\item[(M)] For each $f \in \mathds{F}$, the family of functions
\begin{align}
(z_0,u_0) \mapsto \int_{\mathds{X}} f(x_1) Q(A|x_1,u_0) \mathcal{T}(dx_1|z_0,u_0) \nonumber
\end{align}
is equicontinuous when indexed by $A \in \B(\mathds{Y})$.
\end{itemize}
Using Lemma~\ref{kernel_robust}, it is fairly straightforward to prove that conditions in Theorem~\ref{TV_channel_thm} and Theorem~\ref{TV_kernel_thm} both imply the assumption (M). Hence, assumption (M) is more general than those in Theorem~\ref{TV_channel_thm} and Theorem~\ref{TV_kernel_thm}.

\begin{thm}\label{ext_thm}
Under assumption (M), the transition probability $\eta(\cdot|z,u)$ of the filter process is weakly continuous in $(z,u)$.
\end{thm}

\begin{proof}
Recall that it is sufficient to prove the following:
\begin{itemize}
\item[(i)] $P(dy_1|z_0,u_0)$ is continuous in total variation,
\item[(ii)] $\lim_{n\rightarrow\infty} \int_{\mathds{Y}} \rho(z_1(z_0^n,u_n,y_1),z_1(z_0,u,y_1))P(dy_1|z_0,u) = 0$ as $(z_0^n,u_n) \rightarrow (z_0,u)$.
\end{itemize}

Firstly, (i) is true since $1_{\mathds{X}} \in \mathds{F}$. For (ii), it suffices to show that
\begin{align*}
&\lim_{n\rightarrow\infty} \int_{\mathds{X}} f_m(x_1) Q(I_+^{(n)}|x_1)\mathcal{T}(dx_1|z_0^n,u_n) -\int_{\mathds{X}} f_m(x_1) Q(I_+^{(n)}|x_1) \mathcal{T}(dx_1|z_0,u) = 0.
\end{align*}
But this immediately follows from assumption (M).
\end{proof}

A careful look at the proof of the weak Feller property of the filter process in Feinberg et. al. \cite{FeKaZg14} reveals that they have first established the weak Feller property under a condition somewhat similar to the assumption (M), and then, establish Theorem~\ref{TV_channel_thm} by proving that assumptions in Theorem~\ref{TV_channel_thm} imply this more general condition. Indeed, let $\tau_b = \{O_j\} \subset \mathds{X}$ be a countable base for the topology on $\mathds{X}$ such that $\mathds{X} \in \tau_b$. Then, under the following assumption:
\begin{itemize}
\item[(F)] For each finite intersection $O = \bigcap_{n=1}^N O_{j_n}$, where $O_{j_n} \in \tau_b$, the family of functions
\begin{align}
\hspace{-20pt} (z_0,u_0) \mapsto \int_{\mathds{X}} \int_{\mathds{X}} 1_O(x_1) Q(A|x_1,u_0) \mathcal{T}(dx_1|x_0,u_0) z_0(dx_0) \nonumber
\end{align}
is equicontinuous when indexed by $A \in \B(\mathds{Y})$,
\end{itemize}
they have proved that the weak Feller property of the filter process holds (see \cite[Lemma 5.3]{FeKaZg14} and \cite[Theorem 5.5]{FeKaZg14-(b)}). We observe that (F) is very similar to (M) except that, in (F), Feinberg et. al. use open sets instead of continuous and bounded functions. However, proving that conditions in Theorem~\ref{TV_channel_thm} imply the assumption (F) as in \cite{FeKaZg14} requires quite tedious mathematical methods.  By using open sets instead
of continuous and bounded functions, one needs to work
with inequalities and limit infimum operation as a result of
Portmanteau theorem [4, Theorem 2.1] (and the associated proof program involving generalized Fatou's lemma \cite{FeKaLi18,FeKaZg_jomaa_16}), in place of equalities
and limit operation, which are significantly easier to analyze than the former leading to a much more concise analysis that we have presented in this paper. For instance, \cite[Theorem 5.1]{FeKaZg14} is the key result to prove that the weak Feller condition of the transition probability and the total variation continuity of the observation channel imply the assumption (F).
We note that if one states this result using continuous and bounded functions in place of open sets, then this version of \cite[Theorem 5.1]{FeKaZg14} becomes a corollary of Lemma~\ref{kernel_robust}, which has a concise and easy to follow proof. But, the proof of \cite[Theorem 5.1]{FeKaZg14} with open sets requires quite tedious mathematical concepts from topology and weak convergence of probability measures. In view of this discussion, we also note that Theorem~\ref{TV_kernel_thm} can also be proved using the condition (F) rather than our approach building on (M) through some additional argumentation. 

In summary, our approach allows for a more direct and concise approach which also makes the proof of Theorem~\ref{TV_channel_thm} more accessible. Once again, we note that Theorem~\ref{TV_kernel_thm} has not been reported in the literature.

\section{Conclusion}

In this paper, there are two main contributions: (i) the weak Feller property of the filter process is established under a new condition, which assumes that the state transition probability is continuous under the total variation convergence with no assumptions on the measurement model, and (ii) a concise and easy to follow proof of the same result under the weak Feller condition of the transition probability and the total variation continuity of the observation channel, which was first established in \cite{FeKaZg14}, is also given. Implications of these results have also been presented.

\section{Acknowledgements}
The authors are grateful to Prof. Eugene Feinberg for generously providing technical feedback and pointing out to related results. 
\section*{Appendix}

\subsection{Proof of Lemma \ref{kernel_robust}}\label{app2}

Note that since, for any $x_n \to x$ in $\mathds{X}$, we have
\begin{align}\label{f2}
&\lim_{n\rightarrow\infty}\sup_{\lambda \in \Lambda}|f_{\lambda}(x_n)-f_{\lambda}(x)| = 0,
\end{align}
we see that $\{f_\lambda\}_{\lambda \in \Lambda}$ is an equicontinuous family of functions. Thus, by the Arzela-Ascoli Theorem \cite{Dud89}, for any given compact set $K \subset \mathds{X}$ and $\epsilon>0$, there is a finite set of continuous and bounded functions $\mathds{F}:=\{f_1,\dots,f_N\}$, so that, for any $\lambda \in \Lambda$, there is $f_i \in \mathds{F}$ with
\[\sup_{x \in K}|f_\lambda(x)-f_i(x) |\leq \epsilon.\]
Now, we claim that, for the same $f_i \in \mathds{F}$, we have $\sup_{x\in K} |f_{n,\lambda}(x)-f_i(x)|\leq 3\epsilon/2$ for large enough $n$, which is independent of $\lambda$. To see this, observe the following:
\begin{align*}
&\sup_{x\in K} |f_{n,\lambda}(x)-f_i(x)|\leq \sup_{x\in K} |f_{n,\lambda}(x)-f_\lambda(x)| +\sup_{x\in K} |f_\lambda(x)-f_i(x)|.
\end{align*}
Note that the second term is less than $\epsilon$ and the first term can be made arbitrarily small as $f_{n,\lambda} \to f_\lambda$ uniformly on compact sets and on $\Lambda$, which can be easily proved using the assumptions in the lemma.

Note that $\mu_n \to \mu$ weakly. Hence, $\{\mu_n\}$ is a tight family of probability measures by Prokhorov theorem \cite[Theorem 5.2]{Billingsley}. Therefore, for any $\epsilon >0$, there exists a compact subset $K_\epsilon$ of $\mathds{X}$ such that, for all $n$,
\[ \mu_n(K_\epsilon) \geq 1-\epsilon.\]

Now, we fix any $\epsilon>0$ and choose a compact set $K_\epsilon$  such that, for all $n$, $\mu_n(K_\epsilon) \geq 1-\epsilon$. We also fix a finite family of continuous and bounded functions $\mathds{F}:=\{f_1,\dots,f_{N}\}$ such that, for any $\lambda$, we can find $f_i \in \mathds{F}$ with $\sup_{x \in K_\epsilon}|f_\lambda(x)-f_i(x) |\leq \epsilon$. Moreover, we choose a large $N$ such that $\sup_{x\in K_\epsilon} |f_{n,\gamma}(x)-f_i(x)|\leq 3\epsilon/2$ for all $n\geq N$.

With this setup, we go back to the main statement:

\begin{align*}
&\sup_{\lambda \in \Lambda} \bigg|\int f_{n,\lambda}(x) \mu_n(dx) - \int f_{\lambda}(x)\mu(dx)\bigg|\\
&\leq\sup_{\lambda \in \Lambda}\bigg|\int_{\mathds{X}\setminus K_\epsilon}f_{n,\lambda}(x) \mu_n(dx) - \int_{\mathds{X}\setminus K_\epsilon}  f_{\lambda}(x)\mu(dx)\bigg|\\
&\phantom{xxxxxxx}+\sup_{\lambda \in \Lambda}\bigg|\int_{ K_\epsilon} f_{n,\lambda}(x) \mu_n(dx)- \int_{ K_\epsilon}  f_{\lambda}(x)\mu(dx)\bigg|\\
&\leq 2\epsilon C + \sup_{\lambda \in \Lambda}\bigg|\int_{K_\epsilon} \big(f_{n,\lambda}(x) -f_i(x)\big) \mu_n(dx) + \int_{K_\epsilon} f_i(x) \mu_n(dx) - \int_{K_\epsilon} f_i(x) \mu(dx)\\
&\qquad\qquad\qquad\qquad + \int_{K_\epsilon} \big(f_i(x)-f_\lambda(x)\big) \mu(dx) \bigg|\\
&\leq 2\epsilon C +  \bigg| \int_{K_\epsilon} f_i(x) \mu_n(dx) - \int_{K_\epsilon} f_i(x) \mu(dx) \bigg| + 5\epsilon/2\\
&\leq 2\epsilon C +  \bigg| \int_{\mathds{X}} f_i(x) \mu_n(dx) - \int_{\mathds{X}} f_i(x) \mu(dx) \bigg|\\
&\phantom{xxxxxxxx}+ \bigg| \int_{K_\epsilon^c} f_i(x) \mu(dx) - \int_{K_\epsilon^c} f_i(x) \mu_n(dx) \bigg| + 5\epsilon/2\\
&\leq 4\epsilon C + 5\epsilon /2 + \bigg| \int_{\mathds{X}} f_i(x) \mu_n(dx) - \int_{\mathds{X}} f_i(x) \mu(dx) \bigg|
\end{align*}
where $C$ is the uniform bound on $\{f_{n,\lambda}\}$ and $\{f_{\lambda}\}$. Since $\epsilon$ is arbitrary and $\mu_n$ converges weakly to $\mu$, the result follows.

\end{document}